\documentclass[a4paper,11pt]{article}
\usepackage{amsmath,amssymb,amsfonts,mathrsfs,hyperref,microtype}
\usepackage{graphicx,color}
\usepackage{ae}
\usepackage{amsthm}

\newtheorem{defn}{Definition}[section]
\newtheorem{thm}{Theorem}[section]
\newtheorem{prop}{Proposition}[section]

\newtheorem{rmk}{Remark}[section]
\newtheorem{lma}{Lemma}[section]
\newtheorem{Assumption}{Assumption}[section]
\newtheorem{exm}{Example}[section]

\newcommand{\ud}{\mathrm{d}}
\newcommand{\la}{\langle}
\newcommand{\ra}{\rangle}

\def\N{{\rm I\kern-0.16em N}}
\def\R{{\rm I\kern-0.16em R}}
\def\E{{\rm I\kern-0.16em E}}
\def\P{{\rm I\kern-0.16em P}}
\def\F{{\rm I\kern-0.16em F}}
\def\B{{\rm I\kern-0.16em B}}
\def\C{{\rm I\kern-0.46em C}}
\def\G{{\rm I\kern-0.50em G}}

\numberwithin{equation}{section}
\font\eka=cmex10
\usepackage{ae}
\def\ind{\mathrel{\hbox{\rlap{%
\hbox to 7.5pt{\hrulefill}}\raise6.6pt\hbox{\eka\char'167}}}}
\begin{document}
\title{Aspects of Stochastic Integration with Respect to Processes of Unbounded p-variation}

\renewcommand{\thefootnote}{\fnsymbol{footnote}}

\author{Zhe Chen\footnotemark[2] \, and \, Lauri Viitasaari\footnotemark[2]}

\footnotetext[2]{Department of Mathematics and System Analysis, Aalto University School of Science, Helsinki P.O. Box 11100, FIN-00076 Aalto,  FINLAND.}

\maketitle


\abstract{This paper deals with stochastic integrals of form $\int_0^T f(X_u)\ud Y_u$ in a case where the function $f$ has discontinuities, and hence the process $f(X)$ is usually of unbounded $p$-variation for every $p\geq 1$. Consequently, integration theory introduced by Young \cite{LCY} or rough path theory introduced by Lyons \cite{lyons} cannot be applied directly. In this paper we prove the existence of such integrals in a pathwise sense provided that $X$ and $Y$ have suitably regular paths together with some minor additional assumptions. In many cases of interest, our results extend the celebrated results by Young \cite{LCY}.

\

\noindent {\bf Keywords}: Unbounded $p$-variation, H\"older processes, Generalised Lebesgue--Stieltjes integral, Riemann-Stieltjes integral

\noindent{\bf MSC 2010:} 60G07, 60H05

\tableofcontents

\section{Introduction}
In this article we extend the notion of pathwise stochastic integral to cover a case $\int_0^T f(X_u)\ud Y_u$, where the processes $X$ and $Y$ are H\"older continuous of order $\alpha_1$ and $\alpha_2$ with $\alpha_1+\alpha_2>1$, and $f$ is a function of locally bounded variation. Consequently, the process $f(X)$ is usually of unbounded variation for every $p\geq 1$ and it is not clear how such integrals can be defined. 

The motivation to study pathwise stochastic integrals arises from many practical cases in different fields of science. For example, many phenomenons in physics seems to follow certain stochastic differential equation, and if the driving process is not a semimartingale it is not clear how to define the stochastic integral. Provided we are given a situation to study, one has to first define how to interpret the stochastic integral. In a semimartingale case the classical It\^o integration theory can be applied. On the other hand, it is widely acknowledged that semimartingales are far from being optimal in order to model different situations, and hence one needs to define different notions of stochastic integration. 

The classical theory of Young \cite{LCY} provides an answer if $f(X)$ and $Y$ are of bounded variation with indexes $p$ and $q$ satisfying $\frac{1}{p}+\frac{1}{q}>1$, and in this case the integral can be defined as a limit of Riemann-Stieltjes sums. This results is particularly used in the literature to study the case where both the integrator and the integrand are H\"older continuous of orders $\alpha$ and $\beta$ with $\alpha+\beta>1$. Furthermore, with the concepts of fractional derivative and integral operators (see \cite{s-k-m}), the integral can be defined in so-called generalized Lebesgue--Stieltjes integrals which was studied in details by Nualart and R\u{a}\c{s}canu\cite{N-R} (see also Z\"ahle \cite{ZM}). In particular, in \cite{N-R} the authors studied stochastic differential equations with smooth coefficients and regular driving process, i.e. a process which is $\alpha$-H\"older continuous with some $\alpha>\frac{1}{2}$. Moreover, such differential equations has been studied extensively in the literature. 

Recently the focus in the literature has been on rough path theory introduced by Lyons \cite{lyons} (this paper is not about rough paths and our aim is not to cover the literature related to rough paths) which provides an integration theory for processes of bounded $p$- and $q$-variations beyond the case $\frac{1}{p}+\frac{1}{q}>1$. Especially, the theory has been successfully applied to study differential equations with smooth coefficients and non-regular driving process, i.e. a process which is $\alpha$-H\"older continuous with some $\alpha<\frac{1}{2}$ (although a restriction $\alpha>\frac{1}{4}$ appears in many of the cases).

To the best of our knowledge there exists not many studies on stochastic integration beyond the mentioned cases, and we believe that the theory should be developed. Indeed, if one considers differential equations arising from phenomenons in science it is not always fitting to assume (sufficiently) smooth coefficients. More importantly, it is not always even the case that the coefficients are continuous.  

Let us now consider an integral of form $\int_0^T f(X_u)\ud Y_u$, where $f$ have discontinuities and $Y$ is not a semimartingale. If $f$ is a (one-sided) derivative of a convex function and $X=Y=B^H$ where $B^H$ is a fractional Brownian motion with Hurst index $H>\frac{1}{2}$, it was proved in Azmoodeh et al. \cite{AMV} that such integrals can be defined as a generalized Lebesgue--Stieltjes integral. More surprisingly, it was proved in \cite{AMV} that the integral can be approximated almost surely with forward sums along uniform partition. Later the theory was further developed in Azmoodeh and Viitasaari \cite{A-V} where the authors studied $L^p$-convergence of such approximation. Moreover, the authors derived the rate of convergence for the approximation. Such result was also generalized in Sottinen and Viitasaari \cite{V-S} for more general H\"older continuous Gaussian processes of order $\alpha>\frac{1}{2}$. Moreover, the authors in \cite{V-S} proved that such integrals can be defined over any (bounded) random interval and the integrals can be approximated with forward sums along any partition. 

In this article we extend the results presented in \cite{AMV} and \cite{V-S}, and show that above mentioned integrals can be defined for much more general class of processes, not just Gaussian processes. Moreover, in \cite{V-S} the authors posed four conditions for the covariance of the corresponding Gaussian process. While such assumptions were not very restrictive, the results of this paper enables us to drop two of the extra conditions. We also study the existence of the integral as a limit of any Riemann-Stieltjes sums, not only the forward sums which are commonly used in the theory of stochastic integration. As a consequence, we also derive change of variable formula and give some basic results on integration with respect to processes of unbounded $p$-variation. Finally, but maybe most importantly, there is a gap in the proof of It\^o formula in \cite{AMV} (explained in the next section), and same argument was later applied, but not stated, in \cite{V-S}. While the result is intuitively clear and believable, it would be worthy to give rigorous proof. This is also done in the present paper.

Besides our main result we also present some generalisations and additional results which are important for many applications. More precisely, we prove existence of mixed integrals and It\^o formula in particular cases. We also discuss financial implications and further aspects of research.

The rest of the paper is organized as follows. In section \ref{sec:aux} we introduce basic auxiliary facts needed to establish our main results. We also explain the untrue statement present in \cite{AMV}, and explain the contradictions which would be an outcome of such results. The main sections of the paper are \ref{sec:gls} and \ref{sec:rs}. In section \ref{sec:gls} we prove the existence of the integral. Moreover, we prove the existence of the mixed integral and establish It\^o formula for our case. Multidimensional extensions are discussed. Section \ref{sec:rs} is devoted to approximation with Riemann-Stieltjes sums, where we also present integration by parts formula. We end the paper with discussions in section \ref{sec:dis}.

\section{Auxiliary Facts}
\label{sec:aux}
\subsection{Fractional operators and generalized Lebesgue--Stieltjes Integral}

The generalized Lebesgue--Stieltjes integral is based on fractional integration and fractional Besov spaces. For details on fractional integration we refer to \cite{s-k-m} and for fractional Besov spaces we refer to \cite{N-R}.

We first recall the definitions for fractional Besov norms and Lebesgue--Liouville fractional integrals and derivatives.

\begin{defn}
Fix $ 0 <\beta < 1 $.
\begin{enumerate}
\item
The \emph{fractional Besov space} $W^{\beta}_1 =  W^{\beta}_1 ([0,T])$ is the space of real-valued measurable functions $ f :[0,T] \to \mathbb{R}$ such that
\begin{equation*}
{\Vert f \Vert}_{1,\beta} = \sup _{0 \le s < t \le T} \left( \frac{|f(t) - f(s)|}{(t-s)^\beta} + \int _{s}^{t} \frac{|f(u) - f(s) |}{(u-s)^ {1+\beta }} \,\ud u \right) < \infty.
\end{equation*}
\item
The \emph{fractional Besov space} $W^{\beta}_2 =  W^{\beta}_2 ([0,T])$ is the space of real-valued measurable functions  $ f :[0,T] \to \mathbb{R}$ such that
\begin{equation*}
{\Vert f \Vert}_{2,\beta} = \int_{0}^{T} \frac{|f(s)|}{s^ \beta}\, \ud s + \int_{0}^{T}\int_{0}^{s} \frac{|f(u) - f(s) |}{(u-s)^ {1+\beta }} \,\ud u \ud s < \infty.
\end{equation*}
\end{enumerate}
\end{defn}

\begin{rmk}\label{r:rmk1}
Let $C^{\alpha}=C^{\alpha }([0,T])$ denote the space of H\"{o}lder continuous functions of order $\alpha$ on $[0,T]$ and let $ 0< \epsilon < \beta \wedge (1- \beta)$. Then 
\begin{center}
$C^{\beta + \epsilon} \subset W^{\beta}_{1} \subset C^{\beta - \epsilon} 
\quad\mbox{and}\quad C^{\beta + \epsilon} \subset W^{\beta}_{2}$.
\end{center}
\end{rmk}

\begin{defn}
Let $t\in[0,T]$. The \emph{Riemann--Liouville fractional integrals} $I^\beta _{0+}$ and $I^\beta_{t-}$ of order $\beta > 0$ on $[0,T]$ are
\begin{eqnarray*}
(I^\beta _{0+} f)(s) &=& \frac{1}{\Gamma (\beta)} \int _0^s f(u) (s-u)^{\beta -1} \,\ud u, \\
(I^\beta _{t-} f)(s) &=& \frac{(-1)^{-\beta}}{\Gamma (\beta)} \int_s^t f(u) (u-s)^{\beta -1} \,\ud u, 
\end{eqnarray*}
where $\Gamma$ is the Gamma-function.  The \emph{Riemann--Liouville fractional derivatives} $D^{\beta}_{0+}$ and $D^{\beta}_{t-}$ are the left-inverses of the corresponding integrals $I^\beta_{0+}$ and $I^\beta_{t-}$. They can be also define via the \emph{Weyl representation} as
\begin{eqnarray*}
(D^{\beta}_{0+} f)(s) &=& 
\frac{1}{\Gamma(1-\beta)} \left( \frac{f(s)}{s^\beta} + \beta \int_{0}^{s}\frac{f(s) - f(u)}{(s-u)^{\beta + 1}}\,\ud u \right),\\
(D^{\beta}_{t-} f)(s) &=& 
\frac{(-1)^{-\beta}}{\Gamma(1-\beta)} \left( \frac{f(s)}{(t-s)^\beta} + \beta \int_{s}^{t}\frac{f(s) - f(u)}{(u-s)^{\beta + 1}} \,\ud u \right)
\end{eqnarray*}
if $f\in I^\beta_{0+}(L^1)$ or $f\in I^\beta_{t-}(L^1)$, respectively.
\end{defn}

Denote $g_{t-}(s) = g(s)-g(t-)$. 

The generalized Lebesgue--Stieltjes integral is defined in terms of fractional derivative operators according to the next proposition.

\begin{prop}\cite{N-R}\label{pr:n-r}
Let $0<\beta<1$ and let $f \in  W^{\beta}_2$ and $ g \in W^{1- \beta}_1$. Then for any $t \in(0,T]$ the \emph{generalized Lebesgue--Stieltjes integral} exists as the following Lebesgue integral
$$
\int_0^t f(s)\, \ud g(s) =
\int_{0}^{t} (D^{\beta}_{0+} f)(s) (D^{1- \beta}_{t-} g_{t-} )(s) \,\ud s
$$
and is independent of $\beta$.
\end{prop}

\begin{rmk}\label{rmk:coincide}
It is shown in \cite{ZM} that if $f \in C^{\gamma}$ and $g \in C^{\eta}$ with $ \gamma + \eta > 1$, then the generalized Lebesgue--Stieltjes integral $ \int_{0}^{t} f(s) \,\ud g(s)$ exists and coincides with the classical Riemann--Stieltjes integral, i.e., as a limit of Riemann--Stieltjes sums.  This is natural, since in this case one can also define the integrals as \emph{Young integrals} \cite{LCY}. 
\end{rmk}

We will also need the following estimate in order to prove our main theorems.

\begin{thm}\cite{N-R}\label{t:n-r}
Let $ f \in  W^{\beta}_2$ and  $ g \in W^{1- \beta}_1$. Then we have the estimate
$$
\left|\int_{0}^t f(s) \,\ud g(s) \right| \le \sup_{0\le s < t \le T} \big| D^{1 - \beta}_{t-} g_{t-}(s) \big| {\Vert f \Vert}_{2,\beta}.
$$
\end{thm}

To conclude the section we also recall \emph{Garsia--Rademich--Rumsey inequality} (see \cite{N-R} and \cite{g-r-r}).

\begin{lma}
\label{lma:GRR}
Let $p\geq 1$ and $\alpha>\frac{1}{p}$. Then there exists a constant $C=C(\alpha,p)>0$ such that for any continuous function $f$ on $[0,T]$, and for all $0\leq s,t\leq T$ we have
$$
{|f(t)-f(s)|}^p \leq C T^{\alpha p -1}|t-s|^{\alpha p  -1}\int_0^T\!\int_0^T \frac{|f(x)-f(y)|^p}{|x-y|^{\alpha p +1}}\,\ud x\ud y.
$$
\end{lma}

\subsection{Review of pathwise integrals}
\begin{defn}
The sequence of points $\pi_n = \{0=t_0^n < t_1^n < \ldots < t_{k(n)}^n = T\}$ on the interval $[0,T]$ is called the partition of the interval $[0,T]$, and the 
size of the partition is defined as
$$
|\pi_n| =\max_{1\leq j\leq k(n)}|t_j^n -t_{j-1}^n|.
$$
\end{defn}
The $p$-variation of a function $f$ along partition $\pi_n$ is defined as
\begin{equation*}
v_p(f;\pi_n) = \sum_{t_k\in \pi_n} |\Delta f_{t_k}|^p,
\end{equation*}
where $\Delta f_{t_k} = f_{t_k} - f_{t_{k-1}}$.
\begin{defn}
Let $f:[0,T]\mapsto \R$ be a function.
\begin{enumerate}
\item
If the limit 
\begin{equation*}
v_p^0(f) = \lim_{|\pi_n|\rightarrow 0} v_p(f;\pi_n)
\end{equation*}
exists, we say that $f$ has finite $p$-variation.
\item
If
\begin{equation*}
v_p(f) = \sup_{\pi_n} v_p(f;\pi_n) < \infty,
\end{equation*}
where the supremum is taken over all possible integers $n$ and partitions $\pi_n$, we say that $f$ has bounded $p$-variation. 
\end{enumerate}
\end{defn}
We denote by $\mathcal{W}_p([0,T])$ the class of functions with bounded $p$-variation on $[0,T]$ and we equip this class with a norm
\begin{equation*}
||f||_{[p]} := (v_p(f))^{\frac{1}{p}} + ||f||_{\infty},
\end{equation*}
where $||f||_{\infty} = \sup_{0\leq t\leq T}|f(t)|$. It is known that the space $(\mathcal{W}_p, ||\cdot||_{[p]})$ is a Banach space.

A contributing work by Young \cite{LCY} extended classical Riemann-Stieltjes to cover functions of unbounded variation. More precisely, he noticed that $p$-variations can be useful to define integrals. For proofs see also \cite{ru-va}.
\begin{thm}
Let $f\in\mathcal{W}_p([0,T])$ and $g\in\mathcal{W}_q([0,T])$ for some $1\leq p,q< \infty$ with
$\frac{1}{p} + \frac{1}{q} > 1$. Moreover, assume that $f$ and $g$ have no common points of discontinuities. 
Then for any interval $[s,t]\subset[0,T]$ the integral
\begin{equation*}
\int_s^t f\ud g
\end{equation*}
exists as a Riemann-Stieltjes integral.
\end{thm}
The above theorem is particularly useful in the case of H\"older continuous functions. This is the topic of the next theorem (see also Z\"ahle \cite{ZM}).
\begin{thm}
Let $f\in C^{\alpha}([0,T])$ and $g\in C^{\beta}([0,T])$. If $\alpha+\beta>1$, then for any interval $[s,t]\subset[0,T]$ the integral
\begin{equation*}
\int_s^t f \ud g
\end{equation*}
exists as a Riemann-Stieltjes integral.
\end{thm}
In applications such as financial mathematics, it is a wanted feature to define stochastic integrals 
as a limit of Riemann-Stieltjes sums, or so-called forward integrals. 
Such pathwise integrals were also studied by F\"ollmer \cite{Follmer} (see also Sondermann \cite{Sondermann}).
\begin{defn}
Let $(\pi_n)_{n=1}^{\infty}$ be a sequence of partitions $\pi_n=\{0=t_0^n<\ldots<t_{k(n)}^n=T\}$ such 
that $|\pi_n|=\max_{j=1,\ldots,k(n)}|t_j^n-t_{j-1}^n|\rightarrow 0$ as $n\to\infty$ and let $X=(X_t)_{t\in[0,T]}$ be a continuous process.
The F\"{o}llmer integral of a process $Y$ with respect to $X$ over interval $[0,t]$ along the sequence $\pi_n$ is defined as
\begin{equation*}
\int_0^t Y_s \ud X_s = \lim_{n\rightarrow\infty} \sum_{t_j^n\in\pi_n \cap(0,t]}Y^n_{t_{j-1}}\left(X_{t_j^n}-X_{t_{j-1}^n}\right) 
\end{equation*}
if the limit exists almost surely. The integral over the whole interval $[0,T]$ is defined as
\begin{equation*}
\int_0^T Y_s \ud X_s = \lim_{t\rightarrow T} \int_0^t Y_s \ud X_s.
\end{equation*}
\end{defn}
\begin{rmk}
\label{rmk:holder_folmer}
If the processes $X$ and $Y$ are H\"{o}lder continuous processes of order $\alpha$ and $\beta$ with $\alpha+\beta >1$, then it can be shown that the F\"{o}llmer integral exists and coincides with the Young integral.
\end{rmk}
We remark that while the definition is very useful for applications, it can sometimes be difficult to show that the F\"{o}llmer integral exists. 
However, in some cases the existence of the F\"{o}llmer integral can be proved. For instance, this is the case for processes $X$ that have finite 
quadratic variation. We first recall the definition of a quadratic variation process.
\begin{defn}
Let $(\pi_n)_{n=1}^{\infty}$ be a sequence of partitions $\pi_n=\{0=t_0^n<\ldots<t_{k(n)}^n=T\}$ such 
that $|\pi_n|=\max_{j=1,\ldots,k(n)}|t_j^n-t_{j-1}^n|\rightarrow 0$ as $n\to\infty$. Let $X$ be a continuous process.  Then $X$ is a \emph{quadratic variation process along the sequence $(\pi_n)_{n=1}^\infty$} if the limit
$$
{\la X \ra}_t = \lim_{n\rightarrow\infty} \sum_{t_j^n\in\pi_n \cap(0,t]}\left(X_{t_j^n}-X_{t_{j-1}^n}\right)^2
$$
exists almost surely.
\end{defn}
\begin{lma}\label{ito-follmer}\cite{Follmer}
Let $X$ be a continuous quadratic variation process and let $f\in C^{1,2}([0,T]\times\R)$. Let $0\le s< t\le T$. Then
\begin{eqnarray*}
f(t,X_t) &=& f(s,X_s) + \int_s^t \frac{\partial f}{\partial t}(u,X_u)\, \ud u + \int_s^t \frac{\partial f}{\partial x}(u,X_u)\, \ud X_u \\
& &+ \frac{1}{2}\int_s^t \frac{\partial^2 f}{\partial x^2}(u,X_u)\, \ud\langle X\rangle_u.  
\end{eqnarray*}
In particular, the Föllmer integral exists and has a continuous modification.
\end{lma}
For the proof and details, see \cite{Follmer} or \cite{Sondermann}.
\begin{rmk}
Note that in the above result the existence of stochastic integral is a consequence of the existence of other terms. Hence the existence of the stochastic integral is not proved directly but it is rather a consequence of the It\^o formula. In this paper we prove the existence of stochastic integrals directly, which in turn implies the existence of the term involving local times arising from a non-trivial quadratic variation.
\end{rmk}
\subsection{Facts on functions of bounded variation}
\label{subsec:bv}
We first recall some basic results on convex functions. Recall first that for every convex function $f:\R\rightarrow\R$, the left-sided derivative $f'_-$ and right-sided derivative $f'_+$ exists. Moreover, it is known (see, e.g. \cite{R-Y}) that $f'_-$ ($f'_+$, respectively) is increasing and left-continuous (right-continuous, respectively), and the set $\{x: f'_-(x)\neq f'_+(x)\}$ is at most countable. We also recall the following theorem.
\begin{thm}\label{measurethm}
Let $f$ be a convex function. Then the second derivative $f''$ of $f$ exists in the sense of distributions and it is a positive Radon measure. Conversely, for a given Radon measure $\mu$ on real line there exists a convex function $f$ such that $f''=\mu$. Moreover, for any interval $I\subset\R$ and any $x\in\textit{int}(I)$ we have
\begin{equation}
\label{rep_convex_aux}
f'_-(x) = C_I + \frac{1}{2}\int_I sgn(x-a)\mu(\ud a),
\end{equation}
where 
\begin{equation*}
sgn(x)= \begin{cases}
1, & \quad x>0\\
-1, &\quad x\leq 0
\end{cases}.
\end{equation*}
\end{thm}
Note that if the support of the measure $\mu$, denoted by $\textit{supp}(\mu)$, is compact, then the representation (\ref{rep_convex_aux}) holds uniformly with some general constant $C$. This is one of the key facts in order to prove our main results. 

We also recall the following crucial facts for our analysis.
Let $f:\R\rightarrow\R$ be a convex function such that the corresponding measure $\mu$ has compact support and $\phi$ be a positive $C^{\infty}$-function with compact support in $(-\infty,0]$ s.t. $\int_\R \phi(x)\ud x=1$. Defining
$$
f_n(x) = n\int_{-\infty}^0 f(x+y)\phi(ny)\ud y,\quad n\in \N
$$
it is known that for every $n\in\N$, $f_n \in C^{\infty}$ and $f_n$ is locally bounded convex function with the properties that $f_n$ converges to $f$ pointwise and $f'_n$ increases to $f'_-$ \cite{R-Y}. More importantly, for every $g\in C^1_0$ we have
\begin{equation}
\label{convex_appro_smooth}
\lim_{n\rightarrow\infty}\int_\R g(x) f''_n(x)\ud x = \int_\R g(x) \mu(\ud x).
\end{equation}
As a direct consequence of equation (\ref{convex_appro_smooth}) we obtain the following Lemma.

\begin{lma}\label{Lemmamain}
Let $X=(X_t)_{t \geq 0}$ be any non-trivial stochastic process, and let $\mu$ be a positive Radon measure with compact support and $f''_n$ be the corresponding approximation satisfying (\ref{convex_appro_smooth}). Then 
\begin{equation*}
\int \mathbf{1}_{X_s<a<X_t} \mu(da) \leq \lim\inf_n \int \mathbf{1}_{X_s<a<X_t}f''_n(a)\ud a.
\end{equation*}
\end{lma} 

\begin{proof}
The proof follows similar ideas as the proof of Portmanteau theorem concerning convergence of probability measures (see, e.g. \cite{klenke}).

Let a function $h$ be continuous with compact support, and denote $\alpha = \sup_x  |h(x)|$. Suppose there exists a sequence $g_i \in C_0^\infty$ with $-\alpha \leq g_i \leq h$, and $\lim_{i \rightarrow \infty}g_i(x)=h(x)$. Since $f_n^{''}$ is positive, then
\begin{equation*}
\liminf_n\int_{\mathbb{R}} h(x)f_n^{''}(x)dx \geq \liminf_n\int_{\mathbb{R}} g_i(x)f_n^{''}(x)dx=\int_{\mathbb{R}} g_i(x)\mu(dx).
\end{equation*}
Apply dominated convergence theorem to $g_i(x)+\alpha$ and $h(x)+\alpha$ to get
\begin{equation*}
\lim_{n \rightarrow \infty}\int_{\mathbb{R}} g_i(x)\mu(dx)=\int_{\mathbb{R}} h(x)\mu(dx).
\end{equation*}
Therefore we get
\begin{equation*}
\liminf_n\int_{\mathbb{R}} h(x)f_n^{''}(x)dx \geq \int_{\mathbb{R}} h(x)\mu(dx).
\end{equation*}
By using same for $-h$ we obtain
\begin{equation*}
\limsup_n\int_{\mathbb{R}} h(x)f_n^{''}(x)dx \leq \int_{\mathbb{R}} h(x)\mu(dx).
\end{equation*}
Thus
\begin{equation}\label{eq3}
\lim_{n \rightarrow \infty}\int_{\mathbb{R}} h(x)f_n^{''}(x)dx=\int_{\mathbb{R}} h(x)\mu(dx),
\end{equation}
for $h$ continuous with compact support.

Let now $A$ be a compact set. For non-empty open set $G \in A$, define $l_N=1 \wedge (N\cdot d(x, G^c))$. Then $l_N$ is a continuous function with compact support, and $0 \leq l_N \leq \mathbf{1}_G$, $l_N \uparrow \mathbf{1}_G$.

Now for an open set $G=(a, b)$, we can rewrite $l_N$ as
\begin{equation*}
l_N=N(x-a)\mathbf{1}_{a<x < \frac{1}{N}+a}+\mathbf{1}_{\frac{1}{N}+a \leq x \leq b-\frac{1}{N}}+N(b-x)\mathbf{1}_{b-\frac{1}{N} < x < b}.
\end{equation*}
It's obvious that we can find a sequence of non-negative smooth functions $g_n$ with compact support approaching to $l_N$ from below.

Thus we can use the equation (\ref{eq3}) above for $h=l_N$ to obtain 
\begin{equation*} 
\lim_{n \rightarrow \infty}\int_{\mathbb{R}} h(x)f_n^{''}(x)dx=\int_{\mathbb{R}} h(x)\mu(dx),
\end{equation*}
and consequently we get
\begin{equation*}
\int_{\mathbb{R}} l_N \mu(dx)= \liminf_n \int_{\mathbb{R}}l_N f^{''}_ndx \leq \liminf_n \int_{\mathbb{R}}\mathbf{1}_G f^{''}_n(x)dx.
\end{equation*}

Let now $N \rightarrow \infty$ to obtain
\begin{equation*}
\mu(G) \leq \liminf_n \int_{\mathbb{R}}\mathbf{1}_Gf^{''}_n(x)dx.
\end{equation*}

This implies
\begin{equation*}
\int_{\mathbb{R}} \mathbf{1}_{X_s<a<X_t} \mu(da) \leq \liminf_n \int_{\mathbb{R}} \mathbf{1}_{X_s<a<X_t}f^{''}_n(a)da.
\end{equation*}

\end{proof}

Next we briefly recall the relation between functions of bounded variation and the derivatives of convex functions.
\begin{defn}
A function $f\in\R\rightarrow\R$ is called to be of bounded variation if $f$ has bounded $p$-variation over the whole real line $\R$ for $p=1$. The space of functions of bounded variation is denoted by $BV(\R)$. The space of functions of locally bounded variation, denoted by $BV^{loc}(\R)$, is a space of functions $f$ which are of bounded variation over every compact set $K\in\R$.
\end{defn}
The fundamental result called Jordan decomposition explains the relation between functions $f\in BV^{loc}(\R)$ and convex functions.
\begin{thm}
A function $f$ is of bounded variation on an interval $[a,b]$ if and only if it can be written as a difference $f=f_1-f_2$, where $f_1$ and $f_2$ are non-decreasing. 
\end{thm}
Let now $f$ and $g$ be convex functions with $f'_-$ and $g'_-$ as their left-sided derivatives. Now $f'_-$ and $g'_-$ are non-decreasing, and consequently the function
$$
h(x) = f'_-(x) - g'_-(x) 
$$ 
is of locally bounded variation. Consequently, to obtain our results it is sufficient to consider given convex function $f$ and its left-sided derivative $f'_-$ from which the general case follows by linearity. Moreover, without loss of generality we can always assume that the corresponding measure $\mu$ has compact support. The general case follows by reduction arguments which are also presented in the proofs.

\subsection{A word of warning}
In this subsection we give a word of warning and briefly explain the gaps present in the literature.

Let now $B^H$ be a fractional Brownian motion with Hurst index $H>\frac{1}{2}$ and let $f$ be a convex function. In \cite{AMV} the authors proved the existence of integrals of form $\int_0^T f'_-(B^H_u) \ud B^H_u$, and applied this fact to prove It\^o formula
\begin{equation}
\label{ito_fbm}
f(B_T^H) = f(0) + \int_0^T f'_-(B^H_u) \ud B^H_u
\end{equation}
for convex function. To obtain such a result, the authors applied classical It\^o formula to the smooth approximation $f_n$ of convex function $f$ explained in previous subsection, and then applied fractional Besov space techniques to prove the convergence of integrals
$$
\int_0^T (f_n)'(B^H_u)\ud B^H_u \rightarrow \int_0^T f'_-(B^H_u)\ud B^H_u.
$$
For such result one needs to find an integrable dominants for the difference $(f_n)' - f'_-$ in terms of the norm $||\cdot||_{2,\beta}$ (see also proof of Theorem \ref{thmito_convex}). In \cite{AMV} it was argued that, starting from equation (\ref{convex_appro_smooth}), one can take any sequence $\phi_\epsilon$ of smooth functions converging uniformly to dirac delta function $\delta_a$ at point $a$ and obtain that
$$
f''_n(a) \approx \int \phi_\epsilon f''_n(x) \ud x \approx \int \phi_\epsilon \mu(\ud x).
$$
Moreover, one have that
\begin{equation}
\label{equation_false}
\int \phi_\epsilon \mu(\ud x) \approx \int \delta_a \mu(\ud x) = \mu(a) < \infty
\end{equation}
which would lead to $\sup_n f''_n(a) < \infty$ uniformly in $n$. However, the statement of equation (\ref{equation_false}) is false in general; if $\mu$ has atom at point $a$, then we obtain 
$$
\int \phi_\epsilon \mu(\ud x) \rightarrow \infty.
$$
On the other hand, by Lebesgue decomposition theorem (see \cite{rudin}) the measure $\mu$ can be decomposed as
$$
\mu = \mu_{AC} + \mu_{SC} + \mu_{SD},
$$
where $\mu_{AC}$ is absolutely continuous with respect to Lebesgue measure, $\mu_{SC}$ is singular continuous part and $\mu_{SD}$ is singular discontinuous part, i.e. $\mu_{SD}$ corresponds to the atoms of the measure $\mu$. Now the statement (\ref{equation_false}) is clearly true for $\mu_{AC}$ and false for $\mu_{SD}$, but it is not clear whether the statement is true for $\mu_{SC}$. On the other hand, if $\sup_n f''_n(a) < \infty$ does indeed hold for some measure $\mu$, then by applying Lemma \ref{Lemmamain} we obtain
$$
\int \textbf{1}_{X_s < a < X_t} \mu(\ud a) \leq C|X_t - X_s|.
$$
Consequently, for every $\alpha$-H\"older continuous process $X$ we would obtain that 
$$
f'_-(X_\cdot) \in W_1^{\alpha-\epsilon} \subset C^{\alpha-2\epsilon}.
$$
In other words, any function of locally bounded variation applied to H\"older continuous process $X$ would still be H\"older continuous. Clearly, such result is true only if $f'_-$ is sufficiently smooth in which case the integral would reduce back to Young integration theory for H\"older continuous processes.

The above mentioned false argument was also applied in \cite{V-S} to generalise the results of \cite{AMV} to more general class of Gaussian processes, although by examining the proof in \cite{V-S} it is clear that the proof is correct provided that $\mu_{SC}([-\epsilon,\epsilon])=0$ for small enough $\epsilon$. We also note that similar techniques was applied in \cite{Heikki}, where the author studied average of geometric fractional Brownian motion and proved certain type of It\^o formula in that case. To obtain the result the author in \cite{Heikki} proved that for a given functional $X_t$ and the approximating sequence $f_n(X_t)$ we have
$$
\E||f_n||_{2,\beta} \rightarrow \E||f||_{2,\beta}.
$$
Then the author applied dominated convergence theorem to obtain the result. However, this is not sufficient to apply dominated convergence theorem. Moreover, it is not even clear whether it holds that $\E||f||_{2,\beta}<\infty$ (see also Remark \ref{rmk:expectation}). Furthermore, similar gap appears in \cite{AMV} where the authors proved that the integral can be approximated with forward sums. There it was proved that there exists an upper bound which converges to an integrable limit. This neither is sufficient in order to apply dominated convergence theorem. 
\section{Existence of gLS-integrals}
\label{sec:gls}
In this section we prove one of our main results; the existence of gLS-integrals. First we give some definitions and technical lemmas. 

We will make the following assumption.
\begin{Assumption}
\label{assumption_density}
Let $X$ be a stochastic process. We assume that for almost every $t\in[0,T]$, the process $X$ has a density $p_t(y)$ and there exists a function $g\in L^1([0,T])$ such that $\sup_y p_t(y) \leq g(t)$.
\end{Assumption}
The following examples should convince the reader that the assumption is not very restrictive.
\begin{exm}
Let $X$ be a stationary stochastic process such that $X_0$ has a density $p(y)$. Then $sup_y p(y) = C < \infty$, and consequently we can take $g(t)=C$.
\end{exm}
\begin{exm}
Let $X$ be a Gaussian process with variance function $V(t)$. Then we have
$$
\sup_y p_t(y) = \frac{1}{\sqrt{2\pi V(t)}}.
$$
Consequently, $X$ satisfies Assumption \ref{assumption_density} provided that $V(t) \geq ct^{2\beta}$ for some $\beta<1$. Especially, this is usually satisfied for every interesting Gaussian process $X$. Indeed, a natural assumption is that $V(t)>0$ for every $t>0$, i.e. there is some randomness involved. On the other hand, for many interesting cases we have $X_0=0$, and thus one has to only study the behaviour of $V(t)$ at zero. Now if $V(t) \leq ct^{2\beta}$ with some $\beta>1$, then the process is H\"older continuous of order $\beta$ around the origin. Hence if $\beta>1$, the process would be constant which is hardly interesting. Similarly, in the limiting case $V(t)\sim t^2$ the process is differentiable in the mean square sense, and consequently one can apply classical integration techniques. 
\end{exm}
\begin{rmk}
In \cite{V-S} the authors studied convex functions and H\"older continuous Gaussian processes of order $\alpha>\frac{1}{2}$ together with additional assumption $V(t) \geq ct^2$, and the proof \cite{V-S} relies on an estimate for probability $\P(X_s<a<X_t)$ with some given level $a$. Moreover, the gap explained previously does not affect the proof present in \cite{V-S} provided that $\mu_{SC}([-\epsilon,\epsilon])=0$ for some $\epsilon>0$. Consequently, in the case of Gaussian processes and many functions $f$ of interest, we have to pose more restrictive condition $V(t)\geq ct^{2\beta}$ for some $\beta<1$. However, clearly such condition is not very restrictive. In particular, fractional Brownian motion with arbitrary Hurst index $H\in(0,1)$ satisfies Assumption \ref{assumption_density}.
\end{rmk}
\begin{exm}
Assume that a process $X$ satisfies Assumption \ref{assumption_density}. If we add a deterministic drift $f(t)$ and consider a process $Y_t = X_t + f(t)$, it is again clear that then the process $Y$ satisfies Assumption \ref{assumption_density}. Hence our results are valid if one adds a suitably regular drift term into the model.
\end{exm}
\begin{exm}
Assume that a process $X_1$ satisfies Assumption \ref{assumption_density} and a process $X_2$ is independent of $X_1$ and has a density. Then the density of a process $Y = X_1 + X_2$ satisfies
\begin{equation*}
p^Y_t = \int p^1_t(z-y)p^2_t(y) \ud y < f(t)\int p^2_t(y) \ud y=f(t).
\end{equation*}
Consequently, the process $Y$ also satisfies the Assumption \ref{assumption_density}.
\end{exm}
In general, Malliavin calculus is a powerful tool to study the existence and smoothness of the density (see, e.g. \cite{nualart}). Moreover, it is known that any random variable lying in some fixed Wiener-chaos admits a density. In particular, a finite sum of such variables admits a density. 

The following Lemma is a version of similar Lemma from \cite{Heikki}, and is one of our key ingredients. 
\begin{lma}\label{lma:assumption}
Let $X$ be a stochastic process such that Assumption \ref{assumption_density} holds. Then for every function $f: [0,T] \rightarrow \R$ and every $\alpha\in(0,1)$ it holds
$$
\E \int_0^T |X_t +f(t)|^{-\alpha}\ud t < C < \infty,
$$
where the constant $C$ is independent of the function $f$.
\end{lma}

\begin{proof}
\begin{equation*}
\E \int_0^T |X_t +f(t)|^{-\alpha}\ud t= \int_0^T \E |X_t +f(t)|^{-\alpha}\ud t.
\end{equation*}
According to Assumption \ref{assumption_density}, we obtain
\begin{equation*}
\begin{split}
\E |X_t +f(t)|^{-\alpha}& = \int_{\R} |y +f(t)|^{-\alpha} p_t(y)dy\\
&=  \int_{f(t)-1}^{f(t)+1} |y +f(t)|^{-\alpha} p_t(y)dy+ \int_{\R \setminus[f(t)-1, f(t)+1]} |y +f(t)|^{-\alpha} p_t(y)dy\\
&\leq g(t) \int_{f(t)-1}^{f(t)+1} |y +f(t)|^{-\alpha} dy +1\\
&= \frac{2}{1-\alpha}\, g(t)+1.
\end{split}
\end{equation*}

Thus we have
\begin{equation*}
\E \int_0^T |X_t +f(t)|^{-\alpha}\ud t \leq \int_0^T  \frac{2}{1-\alpha}\,g(t)dt +T\leq (\frac{2}{1-\alpha}) \lVert g(t) \rVert_{L^1([0,T])}+T.
\end{equation*}
\end{proof}

\begin{rmk}
In \cite{Heikki} the Lemma was proved in special case $f(t) = a$ and it was assumed that $X$ has density $p_t$ for every $t\in[0,T]$. However, it is clear that it is sufficient to assume the existence of density for almost every $t$. This can be handy for some special cases. For example, in \cite{isaacson} the author constructed a martingale which has atoms on certain time points. Consequently, it is clear that on such points the density does not exists. On the other hand, the resulting martingale in \cite{isaacson} inherits the H\"older properties of standard Brownian motion, and hence it could be applied for our results. Similarly, we wish to present the result with arbitrary measurable function $f(t)$ which might be helpful to further generalise our results.
\end{rmk}

Recall that $C^{\alpha}([0,T])$ denotes the space of $\alpha$-H\"older continuous functions on $[0,T]$ and $BV^{loc}$ denotes the space of functions $f:\R\rightarrow\R$ that are of locally bounded variation. 
\begin{defn}
Consider the space $BV^{loc}\left(C^{\alpha}([0,T])\right)$, i.e. the image of $C^{\alpha}$ under mappings $g:\R\rightarrow\R$ such that $g\in BV^{loc}$. More precisely, $f\in BV^{loc}\left(C^{\alpha}([0,T])\right)$ if and only if it can be represented as
$$
f = g(h(x))
$$
for some $g\in BV^{loc}$ and $h\in C^{\alpha}([0,T])$. We will denote $Z\in BV_iC^{\alpha}$ (abbreviation for ''Bounded variation image of $\alpha$-H\"older space'') if $Z\in BV^{loc}\left(C^{\alpha}([0,T])\right)$ and the driving process $X\in C^{\alpha}([0,T])$ satisfies Assumption \ref{assumption_density}, i.e. $Z$ can be represented as
\begin{equation}
\label{rep_BViC}
Z_t = g(X_t)
\end{equation}
for some $g\in BV^{loc}$ and some $X\in C^{\alpha}([0,T])$ satisfying Assumption \ref{assumption_density}.
\end{defn}
\begin{rmk}
Note that we are not assuming that a density of $Z$ itself exists. Indeed, this is rarely the case as can be seen by considering function $g(x) = \textbf{1}_{x>0}$.
\end{rmk}
\begin{rmk}
The authors in \cite{AMV} and \cite{V-S}, motivated by financial applications, considered also smooth and monotonic transformations $g(X)$ of a Gaussian process $X$. It is clear that such transformations are already included to the space $BV_iC^{\alpha}$. Indeed, if the driving process $X$ satisfies Assumption \ref{assumption_density}, then it is clear that a monotone and smooth transformation $g(X)$ also satisfies Assumption \ref{assumption_density}. 
\end{rmk}
For the rest of the paper, the dependence on the interval $[0,T]$ is omitted on the notation unless otherwise specified.

\begin{thm}\label{thm:main}
Let $\alpha\in(0,1)$. Then for every $\epsilon\in(0,1-\alpha)$ we have
$$
BV_iC^{\alpha+\epsilon} \subset W_2^\alpha
$$ 
almost surely.
\end{thm}
\begin{rmk}
Note that under our additional Assumption \ref{assumption_density} we obtain the following relation between spaces:
$$
C^{\alpha+\epsilon} \subset BV_iC^{\alpha+\epsilon} \subset W_2^{\alpha}
$$
i.e. the space $BV_iC^{\alpha}$ is between the fractional Besov space $W_2^\alpha$ and the space of H\"older continuous functions.
\end{rmk}

\begin{proof}
The proof follows ideas used in \cite{AMV,Heikki}.

Let $Z_t=f(X_t) \in BV_iC^{\alpha+\epsilon}$, where $f \in BV^{loc}$ and $X \in C^{\alpha+\epsilon}([0,T]) $ satisfying Assumption \ref{assumption_density}. Then we need to prove $f(X_t) \in W_2^{\alpha}$, i.e. 
$$
\lVert f(X_t) \rVert_{2, \alpha} \rightarrow \infty, \quad a.s.
$$
Assume $\mathcal{K} := supp(\mu)$ is compact. If the support is not compact, define
\begin{equation*}
\mathcal{K}_n=\{\omega \in \Omega|\sup_{t \in [0,T]} |X_t| \in [0,n]\}
\end{equation*}
and a function $f_n$ as
\begin{equation*}
f_n(x)= \begin{cases} f(-n) & \text{if} \quad x <-n,\\
f(x) & \text{if} \quad -n\leq x \leq n, \\
f(n) & \text{if} \quad x>n.
\end{cases}
\end{equation*}
Then the measure $\mu$ associate with $f_n$ has compact support and $f=f_n$ on $\mathcal{K}_n$. Furthermore, we have
$$
\int_0^T f(X_u)\ud Y_u = \int_0^T f_n(X_u)\ud Y_u
$$
on $\mathcal{K}_n$ and $\Omega = \cup_n \mathcal{K}_n$.

First we have
\begin{equation*}
\int_0^T \frac{|f(X_s)|}{s^\alpha} \ud s \leq \sup_{0 \leq s \leq T}|f(X_s)|\int_0^T \frac{1}{s^\alpha} \ud s < \infty, \quad a.s.
\end{equation*}
Then according to the relation of bounded variation function and convex function and Theorem \ref{measurethm}, we have
\begin{equation*}
\begin{split}
&\int_0^T \int_0^t \frac{|f(X_s)-f(X_t)|}{|t-s|^{1+\alpha}} \ud s \ud t\\
=& \int_0^T \int_0^t \frac{|\int_{\mathcal{K}} \text{sgn}(X_s-a)\mu( \ud a)-\int_{\mathcal{K}} \text{sgn}(X_t-a)\mu( \ud a)|}{|t-s|^{1+\alpha}} \ud s \ud t\\
=& \int_0^T \int_0^t \frac{\int_{\mathcal{K}}(\mathbf{1}_{X_s < a<X_t}+\mathbf{1}_{X_t < a<X_s})\mu( \ud a)}{|t-s|^{1+\alpha}} \ud s \ud t
\end{split}
\end{equation*}
We will only consider the case $\mathbf{1}_{X_s< a<X_t}$ since the other can be treated similarly. By Tonelli's theorem we have
\begin{equation*}
 \int_0^T \!\int_0^t \frac{\int_{\mathcal{K}}(\mathbf{1}_{X_s < a<X_t})\mu(\ud a)}{|t-s|^{1+\alpha}}\,\ud s\ud t =\int_{\mathcal{K}}\Big( \int_0^T \!\int_0^t \frac{\mathbf{1}_{X_s < a<X_t}}{|t-s|^{1+\alpha}}\,\ud s\ud t\Big)\mu(\ud a)
\end{equation*}
Now define
\begin{equation*}
T_t(a):=\sup\{u\in [0,t]: X_u=a\},
\end{equation*}
and let the supremum over an empty set be $0$. When $T_t(a)=0$, then $\mathbf{1}_{X_s< a<X_t}=0$, and thus the above integral is $0$. It is trivial that on the set $\{\omega \in \Omega: a < X_t\}$ we have $T_t(a) < t$ a.s. Therefore
\begin{equation*}
\begin{split}
&\int_0^t \frac{\mathbf{1}_{X_s < a<X_t}}{|t-s|^{1+\alpha}} \ud s\\
\leq& \int_0^{T_t(a)} \frac{\mathbf{1}_{ a<X_t}}{|t-s|^{1+\alpha}} \ud s\\
=& \mathbf{1}_{ a<X_t}\frac{(t-T_t(a))^{-\alpha}-t^{-\alpha}}{\alpha}.
\end{split}
\end{equation*}
Let next $p \geq 1$ and $\gamma >\frac{1}{p}$. By Lemma \ref{lma:GRR} there exist a constant $C=C(p, \gamma) >0$, such that for all $0 \leq s,t \leq T$ we have
\begin{equation*}
{|X_t-X_s|}^p \leq C|t-s|^{\gamma p  -1}\int_0^T\!\int_0^T \frac{|X_u-X_v|^p}{|u-v|^{\gamma p +1}}\,\ud v\ud u.
\end{equation*}
Now let $s=T_t(a)$ to obtain
\begin{equation*}
{|X_t-a|}^p \leq C|t-T_t(a)|^{\gamma p  -1}\int_0^T\!\int_0^T \frac{|X_u-X_v|^p}{|u-v|^{\gamma p +1}}\,\ud v\ud u.
\end{equation*}
Recall that $X \in C^{\alpha+\epsilon}([0,T])$ for some $\alpha \in (0,1)$ and $\epsilon \in (0, 1-\alpha)$. Hence by choosing $\gamma < \alpha+\epsilon$ we obtain
\begin{equation}\label{eq8}
\int_0^T\!\int_0^T \frac{|X_u-X_v|^p}{|u-v|^{\gamma p +1}}\,\ud v\ud u
\leq \int_0^T\!\int_0^T C_1(\omega)|t-s|^{(\alpha+\epsilon-\gamma)p-1}\,\ud v\ud u
<\infty.
\end{equation}

Consequently, we have
\begin{equation*}
|t-T_t(a)|^{-\alpha}\leq C^{-\frac{\alpha}{1-\gamma p}}|X_t-a|^{-\frac{\alpha p}{\gamma p-1}}\Big(\int_0^T\!\int_0^T \frac{|X_u-X_v|^p}{|u-v|^{\gamma p +1}}\,\ud v\ud u\Big)^{-\frac{\alpha}{1-\gamma p}},
\end{equation*}

and by \ref{eq8} we obtain
\begin{equation*}
\int_0^T(t-T_t(a))^{-\alpha}\ud t \leq C(\omega) \int_0^T |X_t-a|^{-\frac{\alpha p}{\gamma p-1}}\ud t.
\end{equation*}
Hence, according to Assumption \ref{assumption_density} and Lemma \ref{lma:assumption}, we get
\begin{equation*}
\mathbb{E}\int_0^T(t-T_t(a))^{-\alpha}\ud t \leq C_2 \mathbb{E}\int_0^T |X_t-a|^{-\frac{\alpha p}{\gamma p-1}}\ud t <\infty
\end{equation*}
provided that $\frac{\alpha p}{\gamma p-1} < 1$. Now this is possible with suitable choices of parameters, e.g. by choosing $\gamma = \alpha +\frac{\epsilon}{2}$ and $p > \frac{2}{\epsilon}$. 
To conclude, we have
\begin{equation*}
\int_0^T\int_0^t \frac{\mathbf{1}_{X_s < a<X_t}}{|t-s|^{1+\alpha}} \ud s \ud t \leq C(\omega)\int_0^T |X_t-a|^{-\frac{\alpha p}{\gamma p-1}}\ud t < \infty \quad a.s.
\end{equation*}
which in turn implies 
\begin{equation*}
 \int_0^T \int_0^t \frac{\int_{\mathcal{K}}(\mathbf{1}_{X_s < a<X_t})\mu( \ud a)}{|t-s|^{1+\alpha}} \ud s \ud t <\infty.
\end{equation*}

\end{proof}
\begin{rmk}\label{rmk:expectation}
We remark that we are not claiming that
$$
\E  \int_0^T \int_0^t \frac{\int_{\mathcal{K}}(\mathbf{1}_{X_s < a<X_t})\mu( \ud a)}{|t-s|^{1+\alpha}} \ud s \ud t <\infty
$$
in general. Indeed, our upper bound depends on the random variable representing H\"older constant of the process $X$ which may or may not have finite moments.
\end{rmk}
Next we will state the existence of integrals by using the previous Theorem. 
\begin{thm}\label{mainthm}
Let $X\in BV_iC^{\alpha}$ and $Y\in C^{\gamma}([0, T])$. If $\alpha+\gamma>1$, then the integral
$$
\int_0^T X_s \ud Y_s
$$
exists almost surely in generalized Lebesgue--Stieltjes sense.
\end{thm} 

\begin{proof}
Choose some $\beta\in(1-\gamma,\alpha)$. Then $Y\in W_1^{1-\beta}$ by Remark \ref{r:rmk1} and $X\in W_2^\beta$ by Theorem \ref{thm:main}. Hence the integral is well-defined by Proposition \ref{pr:n-r}.
\end{proof}

The following theorem is a straightforward consequence and the proofs goes analogously to the proof of similar theorem in \cite{V-S}.
\begin{thm}
Let $X\in BV_iC^{\alpha}$ and $Y\in C^{\gamma}$ with $\alpha+\gamma>1$. Let $\tau\leq T$ be a bounded random time. Then the integral
$$
\int_0^\tau X_s \ud Y_s
$$
exists almost surely in generalized Lebesgue--Stieltjes sense.
\end{thm}

\subsection{Multidimensional Processes}

In this section we will show the existence of the stochastic integral in a case where the integrand depends on several processes.

Consider an $n$-dimensional vector $\bar{\alpha} = (\alpha_1,\ldots,\alpha_n)$.

\begin{defn}
Let $X=(X_1,\ldots,X_n)$ be an $n$-dimensional stochastic process on $[0,T]$. We denote $X\in C^{\bar{\alpha}}([0, T])$ if the processes $X_k$ are independent with each other and for every $k$ we have $X_k \in C^{\alpha_k}([0, T])$.
\end{defn}

\begin{thm}\label{multithm}
Let $X\in C^{\bar{\alpha}}([0, T])$ such that every $X_k$ satisfies Assumption \ref{assumption_density} and $Z\in C^{\gamma}([0, T])$. Moreover, let $ f(x_1, \ldots, x_n) \in BV^{loc}$ with respect to every $x_k$. If $\alpha_k+\gamma >1$, for every $k=1, \ldots, n$, then the integral
$$
\int_0^T f(X_t^1, \ldots, X_t^n) \ud Z_s
$$
exists almost surely in generalized Lebesgue--Stieltjes sense.
\end{thm}

\begin{proof}
We will prove the result by induction. We begin with the case $n=2$ to illustrate our approach while case $n=1$ is in fact Theorem \ref{thm:main}.

\emph{Step 1.} When $n=2$, we have $f=f(x,y)$. Choose some $\beta \in (1-\gamma, \min(\alpha_x, \alpha_y))$, then $Z \in W_1^{1-\beta}$. According to Proposition \ref{pr:n-r}, the generalized Lebesgue-Stieltjes integral exists if we have
\begin{equation*}
\lVert f(X_{\cdot}, Y_{\cdot}) \rVert_{2, \beta} < \infty, \quad a.s.
\end{equation*}
Clearly, for the first part in the norm we have
\begin{equation*}
\int_0^T \frac{| f(X_t, Y_t)|}{t^\beta}\ud t \leq \sup_{0\leq t \leq T}|f(X_t, Y_t)|\int_0^T \frac{1}{t^\beta}\ud t < \infty.
\end{equation*}
For the second part in $\lVert f(X_t, Y_t) \rVert_{2, \beta}$, we write
\begin{equation*}
\begin{split}
&| f(X_t, Y_t)- f(X_s, Y_s)|\\
=&| f(X_t, Y_t)- f(X_s, Y_t)+ f(X_s,Y_t)- f(X_s, Y_s)|\\
\leq &| f(X_t, Y_t)-f(X_s, Y_t)|+| f(X_s,Y_t)-f(X_s, Y_s)|.
\end{split}
\end{equation*}
The key idea for the proof is that for each pair above, the other variable is fixed and we treat the function $ f(x, y)$ as a function of one variable, say $x$ while $y$ is fixed. Consequently, we can apply representation (\ref{rep_convex_aux}) with respect to variable $x$ such that the corresponding Radon measure $\mu = \mu_y$ depends on the fixed variable $y$. 

Next we proceed similarly as in the one-dimensional case. First define a set
\begin{equation*}
\Omega_n = \{\omega \in \Omega: \max\Big(\sup_{t \in [0, T]}|X_t|, \sup_{t \in [0, T]}|Y_t|\Big) \in [0,n]\}, \quad n \in \mathbf{N}
\end{equation*}
and an auxiliary function $f_n$ by 
\begin{equation*}
f_n(x,y)=\\ 
\begin{cases}
f(-n,n),   & x<-n, y>n\\
f(-n,-n), & x<-n, y<-n\\ 
f(x,n), & x\in [-n,n], y>n\\
f(x,-n), & x\in [-n,n], y<-n\\
f(x,y), & -n \leq x \leq n, -n \leq y \leq n, \\
f(n,y), & x>n, y\in[-n,n]\\
f(-n,y), & x<-n, y\in [-n,n]\\
f(n,-n), & x>n, y<-n\\
f(n,n),   & x>n, y>n.
\end{cases}
\end{equation*}
Now the measure $\mu$ associated to $f_n$ has compact support, $ f_n=f$ on $\Omega_n$ and $f(x,y)$ is uniformly bounded. Furthermore, $\Omega = \cup_n \Omega_n$ which gives the general case. Consequently, we can assume without loss of generality that $f(x,y)$ is uniformly bounded and that all the corresponding Radon measures have compact support on $[-n,n]$.

For part $| f(X_t,Y_t)- f(X_s, Y_t)|$ we get by applying representation (\ref{rep_convex_aux})
\begin{equation*}
\begin{split}
&| f(X_t, Y_t)- f(X_s, Y_t)|\\
=&  \int_{-n}^n \text{sgn}(X_t-a)\mu_{Y_t}(\ud a)- \int_{-n}^n \text{sgn}(X_s-a)\mu_{Y_t}(\ud a)\\
=&\int_{-n}^n\Big(\mathbf{1}_{X_s <a<X_t}+ \mathbf{1}_{X_t<a<X_s}\Big) \mu_{Y_t}( \ud a).
\end{split}
\end{equation*}

By symmetry, what we need to show is that
\begin{equation*}
\int_0^T \int_0^t \int_{-n}^n\frac{(\mathbf{1}_{X_s < a<X_t})}{|t-s|^{1+\beta}} \, \mu_{Y_t}(\ud a) \ud s \ud t <\infty.
\end{equation*}

Thus by following the arguments of the proof of Theorem \ref{thm:main} we have
\begin{equation*}
|t-T_t(a)|^{-\beta}\leq C_1^{-\frac{\beta}{1-\gamma p}}|X_t-a|^{-\frac{\beta p}{\gamma p-1}}\Big(\int_0^T\!\int_0^T \frac{|X_u-X_v|^p}{|u-v|^{\gamma p +1}}\,\ud v\ud u\Big)^{-\frac{\beta}{1-\gamma p}},
\end{equation*}
where the double integral is almost surely finite. Hence we only need to show that
\begin{equation}\label{inequ}
\mathbb{E}\int_0^T \int_{-n}^n |X_t-a|^{-\frac{\beta p}{\gamma p-1}} \mu_{Y_t}(\ud a) \ud t <\infty.
\end{equation}

By independence, Assumption \ref{assumption_density} and Lemma \ref{lma:assumption}, we have
\begin{equation*}
\begin{split}
&\mathbb{E}\int_0^T \int_{-n}^n |X_t-a|^{-\frac{\beta p}{\gamma p-1}} \mu_{Y_t}(\ud a) \ud t \\
=&\int_0^T \int_{\R}\int_{-n}^n\E |X_t-a|^{-\frac{\beta p}{\gamma p-1}}\mu_y(\ud a)\P(Y_t \in \ud y)\ud t\\
\leq &\int_0^T \int_{\R}\int_{-n}^n C g(t)\mu_y(\ud a)\P(Y_t \in \ud y)\ud t\\
=&  C\int_0^T g(t) \int_{\R}\int_{-n}^n \mu_y(\ud a)\P(Y_t \in \ud y)\ud t\\
=&  C \int_0^T g(t) \int_{\R}[f(n+\epsilon, y)-f(-n-\epsilon, y)]\P(Y_t \in \ud y)\ud t\\
\leq & C \int_0^T g(t) \ud t\\
<&  \infty,
\end{split}
\end{equation*}
where the third equality comes from the following
\begin{equation*}
\begin{split}
&f(n+\epsilon, y)-f(-n-\epsilon, y)\\
&=\int_{-n}^n \text{sgn}(n+\epsilon-a) \mu_y(\ud a)-\int_{-n}^n \text{sgn}(-n-\epsilon-a) \mu_y(\ud a)\\
=&\int_{-n}^n \mu_y(\ud a) + \int_{-n}^n \mu_y(\ud a)\\
=&2\int_{-n}^n \mu_y(\ud a)
\end{split}
\end{equation*}
and the last inequality comes from the fact
$$
\sup_y |f(n+\epsilon, y)-f(-n-\epsilon, y)| < \infty
$$
since we can assume that $f$ is uniformly bounded in both variables.
This implies
\begin{equation*}
\int_0^T \int_0^t \int_{-n}^n\frac{(\mathbf{1}_{X_s < a<X_t})}{|t-s|^{1+\beta}} \, \mu_{Y_t}(\ud a) \ud s \ud t <\infty.
\end{equation*}

Now turn to the part $| f(X_s,Y_t)- f(X_s, Y_s)|$. By the locally bounded variation property and Fubini's Theorem, we only need to consider 
\begin{equation*}
\int_0^T\int_0^t  \int_{-n}^n\frac{\mathbf{1}_{Y_s < a<Y_t}}{|t-s|^{1+\beta}} \, \mu_{X_s}(\ud a) \ud s \ud t =\int_0^T\int_s^T \int_{-n}^n \frac{\mathbf{1}_{Y_s < a<Y_t}}{|t-s|^{1+\beta}}  \, \mu_{X_s}(\ud a)\ud t \ud s.
\end{equation*}

Now define $T_s(a):=\inf\{t\in [s,T]: Y_t=a\}$. Again by repeating arguments of the proof of Theorem \ref{thm:main} we obtain
\begin{equation*}
|T_s(a)-s|^{-\beta}\leq C_1^{-\frac{\beta}{1-\gamma p}}|a-Y_s|^{-\frac{\beta p}{\gamma p-1}}\Big(\int_0^T\!\int_0^T \frac{|Y_u-Y_v|^p}{|u-v|^{\gamma p +1}}\,\ud v\ud u\Big)^{-\frac{\beta}{1-\gamma p}}.
\end{equation*}
Consequently we only need to show that
\begin{equation*}
\mathbb{E}\int_0^T \int_{-n}^n |Y_s-a|^{-\frac{\beta p}{\gamma p-1}} \mu_{X_s}(\ud a) \ud s <\infty
\end{equation*}
which is actually condition (\ref{inequ}). To summarize, we have proved that 
\begin{equation*}
\begin{split}
&\int_0^T \int_0^T\int_0^t \frac{|f(X_t,Y_t) - f(X_s,Y_s)|}{(t-s)^{\beta+1}}\ud s\ud t\\
&\leq C_T\int_0^T \int_{-n}^n |X_t-a|^{-\frac{\beta p}{\gamma_1 p-1}} \mu_{Y_t}(\ud a) \ud t\\
&+ C_T\int_0^T \int_{-n}^n |Y_t-a|^{-\frac{\beta p}{\gamma_2 p-1}} \mu_{X_t}(\ud a) \ud t
\end{split}
\end{equation*}
for some parameters $\gamma_1$ and $\gamma_2$. Furthermore, by examining the proof we note that the random constant $C_T$ (depending on H\"older constants of $X$ and $Y$) is increasing in $T$. Hence we have 
\begin{equation*}
\begin{split}
&\int_0^v \int_0^t\frac{|f(X_t,Y_t) - f(X_s,Y_s)|}{(t-s)^{\beta+1}}\ud s\ud t\\
&\leq C_T\int_0^v \int_{-n}^n |X_t-a|^{-\frac{\beta p}{\gamma_1 p-1}} \mu_{Y_t}(\ud a) \ud t\\
&+ C_T\int_0^v \int_{-n}^n |Y_t-a|^{-\frac{\beta p}{\gamma_2 p-1}} \mu_{X_t}(\ud a) \ud t
\end{split}
\end{equation*}
for every $v<T$ which in turn implies that
\begin{equation}
\label{eq:induction_step}
\begin{split}
& \int_0^t\frac{|f(X_t,Y_t) - f(X_s,Y_s)|}{(t-s)^{\beta+1}}\ud s \\
&\leq C_T \int_{-n}^n |X_t-a|^{-\frac{\beta p}{\gamma_1 p-1}} \mu_{Y_t}(\ud a)\\
&+ C_T\int_{-n}^n |Y_t-a|^{-\frac{\beta p}{\gamma_2 p-1}} \mu_{X_t}(\ud a),
\end{split}
\end{equation}
where the inequality holds $(\P,\ud t)$-almost surely.

\emph{Step 3.} When $n=3$, we write $f(X^1, X^2, Y)=f(\bar{X}, Y)$. Now we have
\begin{equation*}
\begin{split}
|f(\bar{X}_t, Y_t)- f(\bar{X}_s, Y_s)|= |f(\bar{X}_t, Y_t)- f(\bar{X}_s, Y_t)|+|f(\bar{X}_s, Y_t)- f(\bar{X}_s, Y_s)|.
\end{split}
\end{equation*}
For the latter part, we can apply arguments of Step 2 except that now the Radon measure depends on $\bar{X}_s$.

For the first part, we note that for fixed $\omega$ and $t$ the random variable $Y_t$ is just a number which affects only to the corresponding Radon measures $\mu$. Consequently, by applying (\ref{eq:induction_step} we obtain for the inner integral
\begin{equation*}
\begin{split}
& \int_0^t\frac{|f(\bar{X}_t,Y_t) - f(\bar{X}_s,Y_t)|}{(t-s)^{\beta+1}}\ud s \\
&\leq C_T \int_{-n}^n |X^1_t-a|^{-\frac{\beta p}{\gamma_1 p-1}} \mu_{X^2_t,Y_t}(\ud a)\\
&+ C_T\int_{-n}^n |X^2_t-a|^{-\frac{\beta p}{\gamma_2 p-1}} \mu_{X^1_t,Y_t}(\ud a).
\end{split}
\end{equation*}
Hence the result follows by arguments of Step 2. To conclude, the case $n>3$ follows by repeating the same argument.
\end{proof}
\begin{rmk}
We remark that assumption of independent components $X^k$ is merely a simplification. Indeed, the result remains valid if the conditional variables $X^k$ given the other $n-1$ variables satisfies Assumption \ref{assumption_density}.
\end{rmk}

\subsection{Existence of mixed integrals}
In this section we briefly explain the existence of certain type of mixed integrals which can be particularly important for applications such as mathematical finance. Indeed, consider a process $Y=X+W+J$, where $X$ is H\"older continuous process of order $\alpha>\frac{1}{2}$, $W$ is a standard Brownian motion (or more generally, any local martingale with absolutely continuous quadratic variation), and $J$ is a jump process. Such model can be very useful in different applications and mixed models have received increasing attention in the literature. For example in finance, white noise $W$ represents independent agents acting in the market, $X$ can capture different stylized facts such as long range or short range dependence, and $J$ represents the market shocks.

For such a model the notion of stochastic integral with respect to a martingale $W$ or a jump process $J$ is quite clear; the integral with respect to $W$ is understood in the It\^o sense and the integral with respect to a jump process $J$ is defined in the usual sense
$$
\int_0^T H_s \ud J_s := \sum_{s\leq T} H_{s-}\Delta J_s,
$$
where $\Delta J_s$ denotes the jump at point $s$. The pathwise stochastic integral with respect to $X$ is the topic of the next theorem. 

\begin{thm}\label{mixed}
Let $Y=X+M+J$ be a mixed process, where $X \in C^{\alpha_1}([0,T])$ with $\alpha_1 > \frac{1}{2}$, $M=(M_t)_{t \in [0,T]}$ is a continuous martingale with absolutely continuous quadratic variation and $J=(J_t)_{t \in [0,T]}$ is any process satisfying $\mathbb{P}(J_s \neq J_t) \leq C|t-s|^{\alpha_2}$ with $\alpha_2 > 1-\alpha_1$. Furthermore, assume that $X+M$ satisfies Assumption \ref{assumption_density} and that $J$ is independent of $X+M$. Let $f \in BV^{loc}$, then the integral
$$
\int_0^T f(Y_s) \ud X_s
$$
exists almost surely in generalized Lebesgue--Stieltjes sense.
\end{thm}

\begin{proof}
Note first that since $M$ is a martingale with absolutely continuous quadratic variation, it follows (see \cite{R-Y}) that $M$ is H\"older continuous of order $H$ for any $H<\frac{1}{2}$. Consequently, the process $\tilde{X}=X + M$ is also H\"older continuous of order $H$. Choose now $\beta\in\left(1-\alpha_1,\min\left(\frac{1}{2},\alpha_2\right)\right)$. Consequently, we have
$
\tilde{X} \in W_1^{1-\beta}
$
and it is sufficient to show 
\begin{equation*}
\lVert f(Y_s)\rVert_{2, \beta} < \infty   \quad a.s.
\end{equation*}

It is trivial that
\begin{equation*}
\int_0^T \frac{|f(Y_s)|}{s^\beta}\ud s \leq \sup_{0 \leq s \leq T}|f(Y_s)|\int_0^T \frac{1}{s^\beta} < \infty.
\end{equation*}

Then use representation (\ref{rep_convex_aux}) to have
\begin{equation*}
\begin{split}
\int_0^T \int_0^t \frac{|f(Y_s)-f(Y_t)|}{|t-s|^{1+\beta}}\ud s\ud t&= \int_0^T \int_0^t \frac{|\int_{\mathcal{K}} \text{sgn}(Y_s-a)\mu(\ud a)-\int_{\mathcal{K}} \text{sgn}(Y_t-a)\mu(\ud a)|}{|t-s|^{1+\beta}}\ud s \ud t\\
&= \int_0^T \int_0^t \frac{\int_{\mathcal{K}}(\mathbf{1}_{Y_s < a<Y_t}+\mathbf{1}_{Y_t < a<Y_s)}\mu(\ud a)}{|t-s|^{1+\beta}}\ud s\ud t
\end{split}
\end{equation*}

As before, we only consider $\mathbf{1}_{Y_s < a<Y_t}$ since the other one can be treated similarly. Note first that since $\tilde{X}$ satisfies Assumption \ref{assumption_density} and $J$ is independent of $\tilde{X}$, it follows that the process $Y=X+M+J$ has continuous distribution. Now
\begin{equation*}
\begin{split}
\mathbf{1}_{Y_s< a<Y_t}& = \mathbf{1}_{\tilde{X}_s+J_s <a < \tilde{X}_t+J_t} \\
&=  \mathbf{1}_{\tilde{X}_s< a-J_t<\tilde{X}_t}\mathbf{1}_{J_s=J_t}+\mathbf{1}_{\tilde{X}_s+J_s <a < \tilde{X}_t+J_t}\mathbf{1}_{J_s \neq J_t}\\
&\leq \mathbf{1}_{\tilde{X}_s< a-J_t<\tilde{X}_t}+\mathbf{1}_{J_t \neq J_s}.
\end{split}
\end{equation*}
Hence we get
\begin{equation*}\label{eq2}
\begin{split}
& \int_0^T \int_0^t \frac{\int_{\mathcal{K}}(\mathbf{1}_{Y_s< a<Y_t})\mu(\ud a)}{|t-s|^{1+\beta}}\ud s\ud t\\
&\leq \int_0^T \int_0^t \frac{\int_{\mathcal{K}}\mathbf{1}_{\tilde{X}_s< a-J_t<\tilde{X}_t}\mu(\ud a)}{|t-s|^{1+\beta}}\ud s\ud t+\int_0^T \int_0^t \frac{\int_{\mathcal{K}}\mathbf{1}_{J_s \neq J_t}\mu(\ud a)}{|t-s|^{1+\beta}}\ud s\ud t
\end{split}
\end{equation*}

By conditioning, independence of $J$ and $\tilde{X}$, and the proof of Theorem \ref{thm:main} we obtain that the first integral is finite. For the second integral we take expectation and by Tonelli's Theorem we have
\begin{equation*}
\begin{split}
&\mathbb{E}\int_0^T \int_0^t \frac{\int_{\mathcal{K}}\mathbf{1}_{J_s \neq J_t}\mu(\ud a)}{|t-s|^{1+\beta}}\ud s\ud t\\
&=\int_0^T \int_0^t \frac{\int_{\mathcal{K}}\mathbb{P}(J_s \neq J_t)\mu(\ud a)}{|t-s|^{1+\beta}}\ud s \ud t\\
& \leq C\int_{\mathcal{K}}\int_0^T \int_0^t (t-s)^{\alpha_2-1-\beta}\ud s \ud t \, \mu(\ud a)\\
& < \infty
\end{split}
\end{equation*}
according to our assumption $\mathbb{P}(J_s \neq J_t) \leq C|t-s|^{\alpha_2}$ for some $\alpha_2 > 1-\alpha_1$. 
\end{proof}
\begin{rmk}
A natural choice for the process $J$ is a compound Poisson process given by
$$
J_t = \sum_{k=1}^{N_t} Z_k,
$$
where $N_t$ is a Poisson process and $Z_k$ is any sequence of identically distributed and independent random variables which are also independent of $N_t$. Note however, that $Z_k$:s are not needed to be independent or identically distributed. Similarly, $N_t$ can be replaced with some other jump process. One possible example is fractional Poisson process studied, e.g. by Laskin \cite{laskin}. 
\end{rmk}
\begin{rmk}
Note also that $J$ does not need to be independent of $\tilde{X}$ as long as the conditional variable $\tilde{X}|J$ satisfies Assumption \ref{assumption_density}.
\end{rmk}
\subsection{It\^o formula}
In this section we consider the special case $Y=X$, where $X$ is some $\alpha$-H\"older process with $\alpha>\frac{1}{2}$ satisfying Assumption \ref{assumption_density}, and prove the It\^o formula for such case. We begin with the following smooth version which follows by standard arguments.

\begin{thm}\label{thmito}
Let $X \in C^{\alpha}([0,T])$ with $ \alpha>\frac{1}{2}$, and $f \in C^2(\mathbf{R})$. Then 
$$
f(X_t)=f(X_0)+\int_0^t f'(X_s) \ud X_s.
$$
\end{thm} 

\begin{proof}
By Taylor expansion we have
\begin{equation*}
\begin{split}
 f(X_T)=&f(X_0)+  \int_0^T f'(X_t) \ud X_t+\frac{1}{2} \int_0^T  f''(X_t) \ud [X_t, X_t].
\end{split}
\end{equation*}
Then because $X_t$ has zero quadratic variation, we have
\begin{equation*}
 f(X_T)=f(X_0)+  \int_0^T f'(X_t) \ud X_t.
\end{equation*}

\end{proof}

Next theorem is our main result in this section, and in the proof we fill in the gaps existing in the literature.
\begin{thm}\label{thmito_convex}
Let $X \in C^{\alpha}([0,T])$ satisfy Assumption \ref{assumption_density} with $ \alpha>\frac{1}{2}$, and $f'_- \in BV^{loc}$. Then 
$$
f(X_t)=f(X_0)+\int_0^t f'_-(X_s) \ud X_s,
$$
where $f(x) = \int_0^x f'_-(y)\ud y$.
\end{thm} 
Before the proof we recall the following fundamental result.
\begin{thm}
\label{thm:simple}
Let $X_n,n=1,2,\ldots$ be a positive sequence of integrable random variables and let $X\in L^1(\Omega)$ be a random variable such that $X_n \rightarrow X$ in probability. Then the following are equivalent:
\begin{enumerate}
\item
the sequence $X_n$ is uniformly integrable,
\item
$X_n \rightarrow X$ in $L^1(\Omega)$,
\item
$\E X_n \rightarrow \E X$.
\end{enumerate}
\end{thm}
\begin{rmk}
For the proof the only key feature is the fact that a probability measure $\P$ is finite. Hence it is clear that similar result holds for any finite measure.
\end{rmk}
Consider now a space $S=[0,T]^2$ equipped with a measure $\eta$ given by
$$
\eta(A) = \int_0^T \int_0^T \textbf{1}_A \ud s \ud t.
$$
Clearly, the measure $\eta$ is finite measure. Consequently, it is straightforward to obtain the following version of Theorem \ref{thm:simple}.
\begin{thm}
\label{thm:simple2}
Let $f_n$ be a positive and integrable sequence of functions on $(S,\eta)$ and let $f$ be a positive and integrable function on $(S,\eta)$ such that $f_n \rightarrow f$ for almost every $(s,t)\in S$. Then the following are equivalent;
\begin{enumerate}
\item
the sequence $f_n$ is $\eta$-uniformly integrable,
\item
$f_n \rightarrow f$ in $L^1(\eta)$,
\item
$\int f_n \ud \eta \rightarrow \int f\ud \eta$.
\end{enumerate}
\end{thm}

\begin{proof}[Proof of Theorem \ref{thmito_convex}]
If $f \in C^2$, then we get 
\begin{equation*}
f(X_t)=f(X_0)+\int_0^t f'(X_s) \ud X_s
\end{equation*}
obviously from It\^o formula.
For the general case, let $f_n$ be the smooth approximation of $f$ introduced in chapter \ref{subsec:bv}. Now $f_n$ converges to $f$ pointwise, and $f'_n \in BV^{loc}$. Since $f_n$ is smooth, for $t \in [0,T]$ we get
\begin{equation*}
f_n(X_t)=f_n(X_0)+\int_0^t f'_n(X_s) \ud X_s.
\end{equation*}
Moreover, we can assume that the support of $f''_n$ equals to the support of $\mu$. Indeed, as in the proof of Theorem \ref{thm:main} we can define auxiliary function $f_{n,k}$ for which $f''_{n,k}$ has compact support and $f_{n,k}$ equals to $f_n$ on the set $\{\omega : \sup_{0\leq t\leq T}|X_t| \in [0,n]\}$.

Now it is trivial that $f_n(X_t) \rightarrow f(X_t)$ and $f_n(X_0) \rightarrow f(X_0)$, and hence we only have to prove the convergence of integrals
\begin{equation*}
\int_0^t f'_n(X_s) \ud X_s \rightarrow \int_0^t f'_-(X_s) \ud X_s.
\end{equation*}
By Theorem \ref{t:n-r} it is sufficient to show
\begin{equation*}
\lVert f'_n(X_t)-f'_-(X_t) \rVert_{2, \beta} \rightarrow 0 \quad n \rightarrow \infty \quad a.s.
\end{equation*}
Note first that $f'_n(X_t)-f'_-(X_t)\rightarrow 0$ pointwise. Consequently for the first term in the norm $\lVert \cdot \rVert_{2, \beta}$ we have
\begin{equation*}
\frac{|f'_n(X_t)-f'_-(X_t)|}{t^\beta} \leq \frac{2\sup_{t \in [0,T]}|f'_n(X_t)|}{t^\beta} \in L^1([0,T], dt),
\end{equation*}
which is integrable upper bound, and by Lebesgue dominated convergence theorem we obtain
\begin{equation*} 
\int_0^T \frac{|f'_n(X_t)-f'_-(X_t)|}{t^\beta}  \ud t \rightarrow 0 \quad a.s.
\end{equation*}
Then consider the second term 
\begin{equation*}
\frac{|f'_n(X_t)-f'_-(X_t)-f'_n(X_s)+f'_-(X_s)|}{|t-s|^{1+\beta}} \leq \frac{|f'_n(X_t)-f'_n(X_s)|}{|t-s|^{1+\beta}}+\frac{|f'_-(X_t)-f'_-(X_s)|}{|t-s|^{1+\beta}}.
\end{equation*}
Denote now
$$
g_n(s,t) = \frac{|f'_n(X_t)-f'_n(X_s)|}{|t-s|^{1+\beta}}
$$
and
$$
g(s,t) = \frac{|f'_-(X_t)-f'_-(X_s)|}{|t-s|^{1+\beta}}.
$$
Now for every fixed $\omega$ (outside a set of measure zero) we have $g_n \rightarrow g$. Moreover, sequence $g_n$ is positive and integrable with respect to $\eta$ by H\"older continuity of $X$. Similarly, function $g$ is positive and integrable according to the proof of Theorem \ref{thm:main}. 

Furthermore, since $f_n$ is also convex we can use representation \ref{rep_convex_aux} to get
\begin{equation*}
 \frac{|f'_n(X_t)-f'_n(X_s)|}{|t-s|^{1+\beta}}=\frac{|\int_{\mathcal{K}}(\mathbf{1}_{X_s < a < X_t}+\mathbf{1}_{X_t<a<X_s})f_n''(a)\ud a |}{|t-s|^{1+\beta}}.
\end{equation*}
Hence, by applying Tonelli's theorem, we have 
\begin{equation*}
\begin{split}
&\int_0^T\int_0^t  \frac{|f'_n(X_t)-f'_n(X_s)|}{|t-s|^{1+\beta}}\ud s\ud t \\
=&\int_0^T\int_0^t  \frac{|\int_{\mathcal{K}}(\mathbf{1}_{X_s < a < X_t}+\mathbf{1}_{X_t<a<X_s})f_n''(a)\ud a|}{|t-s|^{1+\beta}}\ud s\ud t\\
=&\int_{\mathcal{K}}\int_0^T\int_0^t \frac{(\mathbf{1}_{X_s < a < X_t}+\mathbf{1}_{X_t<a<X_s})}{|t-s|^{1+\beta}}\ud s\ud t \, f_n''(a)\ud a\\
&\rightarrow \int_{\mathcal{K}}\int_0^T\int_0^t \frac{(\mathbf{1}_{X_s < a < X_t}+\mathbf{1}_{X_t<a<X_s})}{|t-s|^{1+\beta}}\ud s\ud t \, \mu(\ud a)\\
&=\int g \ud \eta,
\end{split}
\end{equation*}
where the convergence takes place according to (\ref{convex_appro_smooth}). Indeed, a function 
$$
a\rightarrow \int_0^T \int_0^t \frac{(\mathbf{1}_{X_s < a < X_t}+\mathbf{1}_{X_t<a<X_s})}{|t-s|^{1+\beta}}\ud s\ud t
$$
is continuous and positive function with compact support. Hence by approximating with smooth functions we obtain the convergence from (\ref{convex_appro_smooth}).
Consequently, we obtained that
$$
\int g_n(s,t)\ud \eta \rightarrow \int g(s,t)\ud \eta
$$ 
which together with Theorem \ref{thm:simple2} implies
$$
\int |g_n - g|\ud \eta \rightarrow 0.
$$
In other words, we have that
$$
\int_0^T \int_0^t \frac{||f_n'(X_t)-f'_n(X_s)|-|f'_-(X_t)-f'_-(X_s)||}{(t-s)^{\beta+1}}\ud s\ud t \rightarrow 0.
$$
To conclude it remains to note that
$$
||f_n'(X_t)-f'_n(X_s)|-|f'_-(X_t)-f'_-(X_s)|| = |f_n'(X_t)-f'_n(X_s)-f'_-(X_t)+f'_-(X_s)|,
$$
since $f'_n$ and $f'_-$ are increasing functions. This gives 
\begin{equation*}
\lVert f'_n(X_t)-f'_-(X_t) \rVert_{2, \beta} \rightarrow 0 \quad n \rightarrow \infty \quad a.s.
\end{equation*}
which concludes the proof.
\end{proof}
\begin{rmk}
We remark that the above It\^o formula is valid even for bounded random times $\tau\leq T$.
\end{rmk}
\begin{defn}
Let $Y$ be a continuous stochastic process with finite quadratic variation $<Y>$. The local time $L^T_a(Y)$ of $Y$ is the process satisfying 
$$
\int_0^T g(Y_u)\ud <Y>_u = \int_\R g(a)L^T_a(Y)\ud a
$$
for every bounded Borel function $g$.
\end{defn}
For continuous mixed processes of form $Y=X+M$ the above convergence of integrals also implies the existence of local time process $L_a^T(Y)$. This is the topic of the next theorem and the proof follows the same lines as the proof of similar theorem in \cite{V-S}. The details are left to the reader.
\begin{thm}
Let $X\in C^\alpha([0,T])$ for some $\alpha>\frac{1}{2}$ and let $M$ be a martingale with absolutely continuous quadratic variation such that $Y=X+M$ satisfies Assumption \ref{assumption_density}. Moreover, let $f$ be a difference of two convex functions. Then the local time process $L_a^T(Y)$ exists and the It\^o-Tanaka formula
$$
f(Y_T) = f(Y_0) + \int_0^T f'_-(Y_u)\ud Y_u + \int_\R L^T_a(Y) \mu(\ud a)
$$
holds almost surely.
\end{thm}

We also have It\^o formula for multidimensional processes. 
\begin{thm}
Let $X\in C^{\bar{\alpha}}([0, T])$ such that for every $k=1,\ldots,n$ the process $X_k$ satisfies Assumption \ref{assumption_density} and $\min_k \alpha_k >\frac{1}{2}$. Furthermore, assume that $\partial_k^- f(x_1,\ldots,x_n) \in BV^{loc}$ with respect to every $x_k$. Then
\begin{equation*}
 f(X_T^1, \ldots, X_T^n)=f(X_0^1, \ldots, X_0^n)+ \sum_{k=1}^n \int_0^T \partial_k^- f(X_t^1, \ldots, X_t^n) \ud X_t^k.
\end{equation*}
 \end{thm}

\begin{proof}
The proof follows similar arguments with the proof of Theorem \ref{thmito_convex} and Theorem \ref{multithm}.

\end{proof}
\section{Existence of Riemann-Stieltjes integrals}
\label{sec:rs}

In this section we prove our second main theorem which states that stochastic integral can also be understood as a limit of Riemann-Stieltjes sums. In \cite{AMV} it was proved that in the particular case of fractional Brownian motion one can approximate such integrals with forward sums along uniform partitions. However, in this case also the use of dominated convergence theorem is a bit floppy. Here we give precise details. Moreover, we prove that the integral exists along any partition and any Riemann-Stieltjes sums.

\begin{thm}\label{R-Sthm}
Let $Z\in BV_iC^{\alpha}$ and $Y\in C^{\gamma}$ with $\alpha+\gamma>1$. Then for any partition $\pi_n=\{0=t_0^n < \ldots < t_{k(n)}^n=T\}$ with $|\pi_n|\rightarrow 0$, we have
$$
\sum_{i=1}^{k(n)}Z_{t_{i}}(Y_{t_{i}^n}-Y_{t_{i-1}^n}) \rightarrow \int_0^T Z_s\ud Y_s \quad a.s.
$$
where $t_i \in [t_{i-1}^n, t_i^n]$.
\end{thm}

\begin{proof}
Without loss of generality we can assume that $Z=f'_-(X)$, where $X\in C^{\alpha}$ and $f'_-$ is the left-derivative of some convex function. Moreover, as in the proof of Theorem \ref{thm:main} we can assume without loss of generality that the measure $\mu$ has compact support. 

First we obtain
\begin{equation*}
\sum_{i=1}^{k(n)}f'_-(X_{t_{i}})(Y_{t_{i}^n}-Y_{t_{i-1}^n})-\int_0^T f'_-(X_s) \ud Y_s=\int_0^T h_n(t) \ud Y_t,
\end{equation*}
where
\begin{equation}
\label{h_n:def}
h_n(t)=\sum_{i=1}^{k(n)}f'_-(X_{t_{i}})\mathbf{1}_{(t_{i-1}^n,t_{i}^n]}(t)- f'_-(X_t).
\end{equation}
Observe that $h_n(t) \rightarrow 0$ pointwise. Hence, by Theorem \ref{t:n-r}, we need to prove that
\begin{equation*}
\lVert h_n(t)\rVert_{2, \beta} \rightarrow 0, \quad a.s.
\end{equation*}
For first term in the norm we have
\begin{equation*}
\frac{|h_n(t)|}{t^\beta} \leq  \frac{2\sup_{t \in [0, T]}|f'_-(X_t)|}{t^\beta},
\end{equation*}
which is integrable, and hence by Lebesgue dominated convergence theorem, we obtain
\begin{equation*} 
\int_0^T \frac{|h_n(t)|}{t^\beta} \ud t \rightarrow 0 \quad a.s.
\end{equation*}
For the second term, we have
\begin{equation*}
\begin{split}
&|h_n(t)-h_n(s)|\\
=&|\sum_{i=1}^{k(n)}f'_-(X_{t_{i}})\mathbf{1}_{(t_{i-1}^n, t_i^n]}(t)-f'_-(X_t)-\sum_{j=1}^{k(n)} f'_-(X_{t_{j}})\mathbf{1}_{(t_{j-1}^n, t_j^n]}(s)+f'_-(X_s)|\\
=&\sum_{i=1}^{k(n)}|f'_-(X_t)-f'_-(X_s)|\mathbf{1}_{(t_{i-1}^n, t_i^n]\times (t_{i-1}^n, t_i^n]}(s,t) \\
&+ \sum_{1 \leq i \leq k(n), j<i }|f'_-(X_{t_{i}})-f'_-(X_{t})-f'_-(X_{t_{j}})+f'_-(X_{s})|\mathbf{1}_{(t_{j-1}^n, t_j^n]\times (t_{i-1}^n, t_i^n]}(s,t)\\
&=: A_1 + A_2.
\end{split}
\end{equation*}
For the term $A_1$, we have
$$
A_1 \leq |f'_-(X_t)-f'_-(X_s)|
$$
which brings an integrable upper bound. Consider now the term $A_2$ where $s$ and $t$ are on different intervals, i.e. $j<i$. Consider first a convex function $f(x) = (x-a)^+$. In this case $f'_-(x) = \textbf{1}_{x>a}$ and we have, almost surely, that
\begin{equation*}
\begin{split}
&|\textbf{1}_{X_{t_i}>a} - \textbf{1}_{X_{t}>a}-\textbf{1}_{X_{t_j}>a}+\textbf{1}_{X_{t_j}>a}|\\
&= \textbf{1}_{X_{t_i}>a>X_t,X_s,X_{t_j}}+\textbf{1}_{X_{t_j}>a>X_t,X_s,X_{t_i}}+\textbf{1}_{X_{t_i}<a<X_t,X_s,X_{t_j}}+\textbf{1}_{X_{t_j}<a<X_t,X_s,X_{t_i}}\\
&+\textbf{1}_{X_{t}>a>X_{t_i},X_s,X_{t_j}}+\textbf{1}_{X_{s}>a>X_t,X_{t_i},X_{t_j}}+\textbf{1}_{X_{t}<a<X_{t_i},X_s,X_{t_j}}+\textbf{1}_{X_{s}<a<X_t,X_{t_i},X_{t_j}}\\
&+2\textbf{1}_{X_{t_i},X_s>a>X_t,X_{t_j}}+2\textbf{1}_{X_{t_i},X_s<a<X_t,X_{t_j}}
\end{split}
\end{equation*}
Now all of the above 12 terms converges to zero almost surely. Let now $\omega$ be fixed. Now on each subinterval we have $\max(|t-t_i|,|s-t_j|)\rightarrow 0$ as $n\rightarrow \infty$. Hence, by H\"older continuity of $X$, there exists $N(\omega)$ such that the first four terms equal to zero for every $n\geq N(\omega)$. Moreover, for last eight terms we have integrable upper bound $\textbf{1}_{X_t<a<X_s} + \textbf{1}_{X_s<a<X_t}$. Hence we obtain that for every $n\geq N(\omega)$ there exists an integrable upper bound, and consequently we get
\begin{equation*}
\int_0^T\int_0^t \frac{|h_n(t)-h_n(s)|}{(t-s)^{\beta +1}}\ud s\ud t \rightarrow 0 \quad a.s. ~\text{as} ~ n \rightarrow \infty
\end{equation*}
by dominated convergence theorem. To conclude the proof, denote by $h_n^a(t)$ the sum given by (\ref{h_n:def}) in the particular case of $f(x)=(x-a)^+$, and denote by $h_n(t)$ the sum given by (\ref{h_n:def}) for general $f$ for which the measure $\mu$ has compact support. By applying representation (\ref{rep_convex_aux}) we obtain
$$
|h_n(t)| \leq \int |h_n^a(t)|\mu(\ud a)
$$
and the result follows by applying dominated convergence theorem (see also \cite{A-V} for details).
\end{proof}

\begin{rmk}
Let $X\in C^{\alpha}$ and $Y\in C^{\gamma}$ with $\alpha+\gamma >1$. In this case, according to celebrated Young results, the integral
 $\int_0^T X_u \ud Y_u$ exists as a limit of Riemann-Stieltjes sums. Under our additional Assumption \ref{assumption_density}, this result is merely a special case corresponding to the function $g(x)=x$ in representation (\ref{rep_BViC}). 
\end{rmk}

Not surprisingly, similar result also holds for multidimensional case. The proof is straightforward and the details are left to the reader.
\begin{thm}
Let $X\in C^{\bar{\alpha}}([0, T])$ such that every $X_k$ satisfies Assumption \ref{assumption_density} and $Y\in C^{\gamma}([0, T])$ with $\alpha_k+\gamma >1$ for $k=1, \ldots, n$. Moreover, assume that $f(x_1, \ldots, x_n) \in BV^{loc}$ with respect to every $x_k$. Then for any partition $\pi_n=\{0=t_0^n < \ldots < t_{k(n)}^n=T\}$ with $|\pi_n|\rightarrow 0$, we have
\begin{equation*}
\sum_{i=1}^{k(n)} f (X^1_{t_i}, \ldots, X^n_{t_i})(Y_{t_i^n}-Y_{t_{i-1}^n}) \xrightarrow{a.s.} \int_0^T f(X^1_t, \ldots, X^n_t)\ud Y_t,
\end{equation*}
where $t_i \in [t_{i-1}^n, t_i^n]$.
\end{thm} 

\subsection{Change of variable formula}
As a direct corollary we obtain the following change of variable formula.
\begin{thm}
\label{thm:change_of_variable}
Let $X\in BV_iC^{\alpha}$ and $Y\in C^{\gamma}$ with $\alpha+\gamma>1$. Then the integral 
$$
\int_0^T Y_u \ud X_u
$$
exists as a limit of Riemann-Stieltjes sums. Moreover, we have the following change of variable formula:
$$
\int_0^T X_u \ud Y_u = Y_T X_T - Y_0X_0 - \int_0^T Y_u \ud X_u.
$$
\end{thm}
\begin{proof}
The existence of the integral and the integration by parts formula are direct consequences of the classical results of functional analysis (see, e.g. \cite{h-p}) combined with Theorem \ref{R-Sthm}.
\end{proof}
\begin{rmk}
To the best of our knowledge, Theorem \ref{thm:change_of_variable} above is one of the first results related to integration with respect to processes of unbounded $p$-variation for every $p\geq 1$. The intuition behind the result is that while the process $X\in BV_iC^\alpha$ is usually of unbounded $p$-variation, the differences $X_{t_k}-X_{t_{k-1}}$ can be either positive or negative, and the compensation leads to finite value of the integral. If one considers the absolute values of the differences, then the Riemann-Stieltjes sum usually diverges. This is studied in details in \cite{azmoodeh} in the case of fractional Brownian motion. 
\end{rmk}

\section{Discussion}
\label{sec:dis}
In this paper we have given some new results related to pathwise stochastic integration for cases when the integrand or the integrator is of unbounded $p$-variation for every $p\geq 1$, and consequently standard results of Young or rough path analysis introduced by Lyons \cite{lyons} cannot be applied. While such results have some important consequences in mathematical finance, it is also worthwhile to note that discontinuities are also present in many fields of other practical applications and the theory should be developed more. For example, in the literature the theory of pathwise (stochastic) integration and especially the theory of stochastic differential equations is mainly based on certain regularity assumptions for the coefficients and the driving process. In particular, most of the literature is focused on the theory of H\"older continuous processes or processes with finite $p$-variation. 

It is also worth of mentioning that in this paper we have, to some extend, concluded the research initiated in \cite{AMV} where similar techniques was applied in the particular case of fractional Brownian motion. More precisely, we have extended the results presented in \cite{AMV} and proved that, together with a careful examination of the proofs together with some ''shortcuts'', same techniques can be applied to any suitably regular processes instead of only fractional Brownian motion (or more general Gaussian processes as in \cite{V-S}). However, for our results we inherited an assumption from the Gaussian case that the underlying process has continuous distribution and a density. As a first extension of our result it would be interesting to study the case where the underlying process has suitably regular path and the corresponding measure has atoms. Moreover, there exists no results corresponding to our case for irregular processes such as H\"older processes of order less than $\frac{1}{2}$ and it is clear that our approach is not fitting for such cases.  

As a final conclusion we remark that our results may have some further implications to mathematical finance. It is clear from the It\^o formula \ref{thmito_convex} that it is easy to construct arbitrage in a model with sufficiently regular driving process with a simple buy-and-sell strategy. Moreover, one can construct strong arbitrage in this case provided that the underlying process satisfies some small deviation estimates (see \cite{m-s-v,she-vii,she-vii2}). On the other hand, it is nowadays well-known that even with geometric fractional Brownian motion the arbitrage disappears under transaction costs \cite{guasoni}. Comparing to our results in section \ref{sec:rs}, we now know that integrals with respect to processes of unbounded variation can be defined which may give rise to some further interesting problems in finance.

\subsection*{Acknowledgments}
Z. Chen was financed by the Finnish Doctoral Programme in Stochastics and Statistics. Both authors thank Ehsan Azmoodeh for his valuable comments which improved the paper.

\bibliographystyle{plain}      
\bibliography{bibli_p2}   

\begin{thebibliography}{10}

\bibitem{azmoodeh}
E.~Azmoodeh.
\newblock On the fractional {B}lack-{S}choles market with transaction costs.
\newblock {\em Communications in Mathematical Finance}, 2(3):21--40, 2013.

\bibitem{AMV}
E.~Azmoodeh, Y.~Mishura, and E.~Valkeila.
\newblock On hedging {E}uropean options in geometric fractional {B}rownian
  motion market model.
\newblock {\em Statistics and Decisions}, 27:129--143, 2010.

\bibitem{A-V}
E.~Azmoodeh and L.~Viitasaari.
\newblock Rate of convergence for discretization of integrals with respect to
  fractional {B}rownian motion.
\newblock {\em J. Theor. Probab.}, DOI:10.1007/s10959--013--0495--y, 2013.

\bibitem{Follmer}
H.~F\"ollmer.
\newblock Calcul d'ito sans probabilit\'es.
\newblock {\em S\'eminaire de probabilit\'es}, 15:143--150, 1981.

\bibitem{g-r-r}
A.~Garsia, E.~Rademich, and H.~Rumsey.
\newblock A real variable lemma and the continuity of paths of some {G}aussian
  processes.
\newblock {\em Indiana Univ. Math. Journal}, 20:565--578, 1970/71.

\bibitem{guasoni}
P.~Guasoni.
\newblock No arbitrage under transaction costs with fractional {B}rownian
  motion and beyond.
\newblock {\em Math. Finance}, 16:569--582, 2006.

\bibitem{h-p}
E.~Hille and R.S. Phillips.
\newblock {\em Functional analysis and semi-groups}.
\newblock American Mathematical Society, 1974.

\bibitem{isaacson}
D.~Isaacson.
\newblock Continuous martingales with discontinuous marginal distributions.
\newblock {\em Annals of Mathematical Statistics}, 42(6):2139--2142, 1971.

\bibitem{klenke}
A.~Klenke.
\newblock {\em Probability Theory}.
\newblock Springer-Verlag, Berlin, 2006.

\bibitem{laskin}
N.~Laskin.
\newblock Fractional {P}oisson process.
\newblock {\em Communications in Nonlinear Science and Numerical simulation},
  8(3-4):201--213, 2003.

\bibitem{lyons}
T.J. Lyons.
\newblock Differential equations driven by rough signals.
\newblock {\em Rev. Mat. Iberoamericana}, 14(2):215--310, 1998.

\bibitem{m-s-v}
Y.~Mishura, G.~Shevchenko, and E.~Valkeila.
\newblock Random variables as pathwise integrals with respect to fractional
  {B}rownian motion.
\newblock {\em Stochastic processes and their Applications}, 123(6):2353--2369,
  2013.

\bibitem{nualart}
D.~Nualart.
\newblock {\em The Malliavin Calculus and Related Topics}.
\newblock Springer, 2006.

\bibitem{N-R}
D.~Nualart and A.~R\u{a}\c{s}canu.
\newblock Differential equations driven by fractional {B}rownian motion.
\newblock {\em Collect. Math.}, 53:55--81, 2002.

\bibitem{R-Y}
D.~Revuz and M.~Yor.
\newblock {\em Continuous martingales and {B}rownian motion}.
\newblock Springer, Berlin, 1999.

\bibitem{rudin}
W.~Rudin.
\newblock {\em Real and Complex Analysis}.
\newblock McGraw-Hill, 1987.

\bibitem{ru-va}
F.~Russo and P.~Vallois.
\newblock Elements of stochastic calculus via regularization.
\newblock {\em S\'eminaire de probabilit\'es}, 40:147--185, 2007.

\bibitem{s-k-m}
S.G. Samko, A.A. Kilbas, and O.I. Marichev.
\newblock {\em Fractional integrals and derivatives: Theory and applications}.
\newblock Gordon and Breach Science Publishers, Yvendon, 1993.

\bibitem{she-vii2}
G.~Shevchenko and L.~Viitasaari.
\newblock Adapted integral representations of random variables.
\newblock {\em arXiv: 1404.7518}, 2014.

\bibitem{she-vii}
G.~Shevchenko and L.~Viitasaari.
\newblock Integral representation with adapted continuous integrand with
  respect to fractional {B}rownian motion.
\newblock {\em Stochastic Analysis and Applications}, To appear, 2014.

\bibitem{Sondermann}
D.~Sondermann.
\newblock {\em Introduction to Stochastic Calculus for Finance: A New Didactic
  Approach}.
\newblock Springer, 2006.

\bibitem{V-S}
T.~Sottinen and L.~Viitasaari.
\newblock Pathwise integrals and {I}t\^o-{T}anaka formula for {G}aussian
  processes.
\newblock {\em submitted}, arxiv:1307.3578, 2013.

\bibitem{Heikki}
H.~Tikanm\"aki.
\newblock Integral representation of some functionals of fractional {B}rownian
  motion.
\newblock {\em Commun.Stoch.Anal.}, 6(2):193--212, 2012.

\bibitem{LCY}
L.C. Young.
\newblock An inequality of the {H}\"older type, connected with {S}tieltjes
  integration.
\newblock {\em Acta. Math.}, 67:251--282, 1936.

\bibitem{ZM}
M.~Z\"ahle.
\newblock Integration with respect to fractal functions and stochastic
  calculus. part {I}.
\newblock {\em Probab. Theory Relat. Fields}, 111:333--372, 1998.

\end{thebibliography}

\end{document}